\documentclass[11pt]{amsart}
\usepackage[margin=1.15in]{geometry}
\usepackage{amscd,amssymb, amsmath, wasysym}
\usepackage{graphicx}
\usepackage{amsfonts}
\usepackage{mathrsfs}    
\usepackage{amsmath}    
\usepackage{amsthm}     
\usepackage{amscd}      
\usepackage{amssymb}    
\usepackage{eucal}      
\usepackage{latexsym}   
\usepackage{graphicx}   
\usepackage{verbatim}   
\usepackage[all]{xy}     

\makeatletter

\newcounter{thmcounter}

\numberwithin{thmcounter}{section}
\numberwithin{equation}{thmcounter}

\newtheorem{theorem}[thmcounter]{Theorem}
\newtheorem{proposition}[thmcounter]{Proposition}
\newtheorem{lemma}[thmcounter]{Lemma}
\newtheorem{corollary}[thmcounter]{Corollary}

\theoremstyle{definition}
\newtheorem{definition}[thmcounter]{Definition}

\newtheorem{remark}[thmcounter]{Remark}

\newtheoremstyle{claim}{9pt}{3pt}{}{\parindent}{\bf}{.}{1em}{}

\theoremstyle{claim}
\newtheorem{claim}[equation]{Claim}

\newenvironment{namelist}[1]{%
\begin{list}{}
{
\settowidth{\labelwidth}{#1}%
\setlength{\labelsep}{0.3em}%
\setlength{\leftmargin}{\labelwidth}%
\addtolength{\leftmargin}{\labelsep}}}{%
\end{list}}

\newcommand{\nR}{\mathbb{R}}                     
\newcommand{\nC}{\mathbb{C}}                     
\newcommand{\nQ}{\mathbb{Q}}                     

\newcommand{\nP}{\mathbb{P}}                     

\newcommand{\sO}{\mathscr{O}}                    

\newcommand{\sI}{\mathscr{I}}                    

\newcommand{\mf}[1]{\mathfrak{#1}}

\DeclareMathOperator{\udiv}{div}                 

\DeclareMathOperator{\Fitt}{Fitt}                

\DeclareMathOperator{\mld}{mld}                  

\DeclareMathOperator{\Supp}{Supp}                

\DeclareMathOperator{\reg}{reg}                  
\DeclareMathOperator{\Exc}{Exc}                  

\newcommand{\ord}{\mbox{ord}}                    

\newcounter{rkcounter}             
\begin{document}

\title[geometric nullstellensatz and symbolic powers]{Geometric nullstellensatz and symbolic powers  on arbitrary varieties}
\author{Wenbo Niu}

\address{Department of Mathematics, Purdue University, West Lafayette, IN 47907-2067, USA}
\email{niu6@math.purdue.edu}



\begin{abstract} In recent years, a multiplier ideal defined on arbitrary varieties, so called Mather-Jacobian multiplier ideal, has been developed independently by Ein-Ishii-Mustata and de Fernex-Docampo. With this new tool, we have a chance of extending some classical results proved in nonsingular case to arbitrary varieties to establish their general forms. In this paper, we first extend a result of geometric nullstellensatz due to Ein-Lazarsfeld in nonsingular case to any projective varieties. Then we prove a result on comparison of symbolic powers with ordinary powers on any varieties, which extends results of Ein-Lazarsfeld-Smith and Hochster-Huneke.

\end{abstract}

\maketitle

\section{Introduction}
Throughout the paper, for simplicity, we work over the complex number field $k:=\nC$. A scheme is always of finite type over $k$ and a variety is a reduced irreducible scheme.

Our first result extends a geometric nullstellensatz due to Ein and Lazarsfeld \cite{Ein:GeomNull}. For the motivation and background on nullstellensatz we refer to the same paper which has a very detailed introduction.
\begin{theorem}\label{p:02} Let $X$ be a projective variety of dimension $n$ with an ample line bundle $L$. Let $\mf{a}$ be an ideal sheaf of $\sO_X$. Assume that $\mf{a}\otimes L$ is globally generated. Then one has
$$\widehat{\sI}(\sO_X)\cdot(\sqrt{\mf{a}})^{n\cdot\deg_LX}\subseteq \mf{a},$$
where the ideal $\widehat{\sI}(\sO_X)$ is the Mather-Jacobian multiplier ideal of $X$.
If assume further that $X$ is normal $\nQ$-Gorenstein, then one can replace $\widehat{\sI}(\sO_X)$ by the usual multiplier ideal $\sI(\sO_X)$ of $X$ in the formula.
\end{theorem}
The ideal $\widehat{\sI}(\sO_X)$ depends only on $X$ and measures its singularities. It is closely related to the Jacobian ideal $\mf{j}_X$ of $X$. Specifically, we have the inclusion $\mf{j}_X\subseteq\overline{\mf{j}}_X\subseteq \widehat{\sI}(\sO_X)$, where $\overline{\mf{j}}_X$ is the integral closure of $\mf{j}_X$. A nullstellensatz of the Jacobian ideal of a variety is also proved in Corollary \ref{p:07}. The proof of theorem follows the spirit of \cite{Ein:GeomNull} utilizing appropriate multiplier ideals. Roughly speaking, we insert an  auxiliary ideal, $\widehat{\sI}(\sO_X)$ or $\sI(\sO_X)$, into the formula obtained by Ein and Lazarsfeld. However, if we put some strong singularity conditions on a variety, then the auxiliary ideal will be trivial so that a clean formula can be obtained (see Corollary \ref{p:08} and Corollary \ref{p:09}).

Our second result gives a comparison of symbolic powers with ordinary powers for a radical ideal on arbitrary varieties. Let $X$ be a variety and  $Z$ be a proper reduced subscheme of $X$ defined by an ideal sheaf $\mf{q}$. Assume that every component of $Z$ has codimension $\leq e$. The $t$-th symbolic power $\mf{q}^t$ is then defined on any affine open set $U$ by $\mf{q}^{(t)}(U)=\{f\in \sO_X(U)| f\in \mf{m}^t_\eta, \mbox{ for all generic point $\eta$ of } Z\}$, where $\mf{m}_\eta$ means the maximal ideal in the local ring $\sO_{X,\eta}$. When $X$ is nonsingular, a surprising result of Ein, Lazarsfeld and Smith \cite{Ein:UniBdSymPw} shows that $\mf{q}^{(me)}\subseteq \mf{q}^m$ for any $m\geq 1$. Based on tight closure theory and reduction to characteristic $p$,  this result has been generalized by Hochster and Huneke \cite{Huneke:CompSymbolicPowers} and in particular when $X$ is singular they showed that $\mf{j}_X^{m+1}\cdot \mf{q}^{(me)}\subseteq \mf{q}^m$ where $\mf{j}_X$ is the Jacobian ideal of $X$.\footnote{Hochster and Huneke considered more general case in the setting of commutative algebra. However, here we stay in our geometric setting of this problem.}  In their paper, Hochster and Huneke also asked if one can drop the power of $\mf{j}_X^{m+1}$ by $1$ to be $\mf{j}_X^m$ \cite[5. Quesitons]{Huneke:CompSymbolicPowers}. Using reduction to characteristic $p$, Takagi investigated this problem in \cite{Takagi:FormSing} by showing that it is the case if $X$ is normal $\nQ$-Gorenstein.\footnote{After the first version of the paper was posted online, Takagi has informed me that his method in fact works without assuming normal $\nQ$-Gorenstein.} Now by using Mather-Jacobian multiplier ideals and inspired by the pioneer work of Ein, Lazarsfeld and Smith, we are able to give a geometric approach to this question and drop the power of the Jacobian ideal as expected.

\begin{theorem}\label{p:04} Let $X$ be a variety and $Z$ be a proper reduced subscheme of $X$ defined by an ideal sheaf $\mf{q}$. Assume that every component of $Z$ has codimension $\leq e$. Then for any $m\geq 1$,
$$\widehat{\sI}(\mf{j}_X^{m-1})\cdot \mf{q}^{(me)}\subseteq \mf{q}^m,$$
where $\mf{j}_X$ is the Jacobian ideal of $X$ and $\widehat{\sI}(\mf{j}_X^{m-1})$ is the Mather-Jacobian multiplier ideal of $\mf{j}_X^{m-1}$.
\end{theorem}

Since we have the inclusion $\mf{j}^m_X\subseteq \overline{\mf{j}^m_X}\subseteq \widehat{\sI}(\mf{j}_X^{m-1})$, where $\overline{\mf{j}^m_X}$ is the integral closure of $\mf{j}^m_X$, the above question has an affirmative answer.
The basic idea to prove this theorem is to use asymptotic Mather-Jacobian multiplier ideals which generalize asymptotic multiplier ideals on nonsingular varieties used in \cite{Ein:UniBdSymPw}. Of course, some technical difficulties should be cleared in order to get the desired result. At this point, the paper \cite{Ein:MultIdeaMatherDis} provides an important basic theorem: subadditivity theorem for Mather-Jacobian multiplier ideals. The only one main difficulty left to us is to relate asymptotic Mather-Jacobian multiplier ideals of symbolic powers to ordinary powers. For certain generalization of the above theorem, see Remark \ref{rmk:02}.

This paper is organized as follows. In the section 2, we briefly review the theory of multiplier ideals and its asymptotic construction. In the section 3, we prove Theorem \ref{p:02} and the section 4 is devoted to Theorem \ref{p:04}.

\vspace{0.3cm}
{\em Acknowledgement}. The author would like to thank professor Lawrence Ein for his inspiring suggestions and professor Shunsuke Takagi for his kindly explaining his work and for his suggestions on Corollary \ref{p:09}. The author's thanks also go to the referee for his/her nice comments which improve the paper.

\section{Multiplier ideals}
In this section, we quickly review the theory of multiplier ideals. Let us start with the classical one defined on a normal $\nQ$-Gorenstein variety. A good reference on this topic is \cite{Lazarsfeld:PosAG2}, though it was stated there mostly in the nonsingular case.

Let $X$ be a variety. Given several ideals $\mf{a}_1,\cdots,\mf{a}_t$ on $X$, a {\em log resolution of singularities} of $X$ and $\mf{a}_1\cdots\mf{a}_t$ is a projective birational morphism $f:Y\longrightarrow X$, with $Y$ nonsingular such that for each $i$, $\mf{a}_i\cdot\sO_Y=\sO_Y(-F_i)$ for an effective divisor $F_i$ on $Y$ and the union of $\Exc(f)$ with those $F_i$'s are simple normal crossings.

A normal variety $X$ is called {\em $\nQ$-Gorenstein}, if $mK_X$ is a Cartier divisor for some integer $m>0$. For any ideal $\mf{a}\subseteq \sO_X$, take a log resolution of $\mf{a}$ as
$f:Y\longrightarrow X$ such that $\mf{a}\cdot\sO_Y=\sO_Y(-E)$ for some effective divisor $E$ on $Y$. An {\em multiplier ideal} of $\mf{a}$ of exponent $t\in \nR_{\geq 0}$ is defined to be
$$\sI(X,\mf{a}^t):=f_*\sO_{X'}(\lceil K_{X'}-f^*K_X-tE\rceil),$$
where $\lceil\ \ \rceil$ means the round up of an $\nR$-divisor.
It can be checked that $\sI(X,\mf{a}^t)$ is an ideal sheaf and is independent on the choice of the log resolution $f$. We also have a local vanishing deduced from Kawamata-Viehweg vanishing theorem, namely, $$R^if_*\sO_{X'}(\lceil K_{X'}-f^*K_X-tE\rceil)=0,\quad \mbox{for }i>0.$$ From this local vanishing, by using Koszul complex on a log resolution, we can deduce the Skoda-type formula: for $l\geq \dim X$, one has $\sI(X,\mf{a}^l)\subseteq a$. Note that these are all we  need in the paper. The multiplier ideal defined here does not have Nadel vanishing theorem in general.

Now we come to the theory of multiplier ideals on arbitrary varieties developed by Ein, Ishii and Mustata in a recent paper \cite{Ein:MultIdeaMatherDis} (see also the work of de Fernex and Docampo \cite{Roi:JDiscrepancy}). We quote some standard definitions directly from \cite{Ein:MultIdeaMatherDis} for the convenience of the reader.

Let $X$ be an $n$-dimensional variety and let $\Omega_X$ be the sheaf of differentials of $X$. Consider the morphism
$$\pi:\nP(\wedge^n\Omega_X)\longrightarrow X$$
which is an isomorphism over the nonsingular locus $X_{\reg}$ of $X$. The {\em Nash blow-up} $\widehat{X}$ is the closure of $\pi^{-1}(X_{\reg})$ in $\nP(\wedge^n\Omega_X)$ with the reduced scheme sturcture. The {\em Jacobian ideal} $\mf{j}_X$ of $X$ is defined to be the $n$-th Fitting ideal of $\Omega_X$, i.e., $\mf{j}_X:=\Fitt^n\Omega_X$.

\begin{definition} Let $X$ be a variety of dimension $n$ and $f:Y\longrightarrow X$ be a resolution of singularities factoring through the Nash blow-up of $X$. Then the image of the canonical homomorphism
$$f^*(\wedge^n\Omega_X)\longrightarrow \wedge^n\Omega_X$$
is an invertible sheaf of the form $J\cdot\wedge^n\Omega_X$ where $J$ is the invertible ideal sheaf on $Y$ that defines an effective divisor which is called the {\em Mather discrepancy divisor} and denoted by $\widehat{K}_{Y/X}$.
\end{definition}

\begin{definition} Let $X$ be a variety and $\mf{a}\subseteq \sO_X$ a nonzero ideals on $X$. Given a log resolution $f:Y\longrightarrow X$ of $X$ and $\mf{j}_X\cdot\mf{a}$ such that $\mf{a}\cdot\sO_Y=\sO_Y(-Z)$ and $\mf{j}_X\cdot\sO_Y=\sO_Y(-J_{Y/X})$ for some effective divisors $Z$ and $J_{Y/X}$ on $Y$ (such resolution automatically factors through the Nash blow-up, see Remark 2.3 of \cite{Ein:MultIdeaMatherDis}). The {\em Mather-Jacobian multiplier ideal} of $\mf{a}$ of exponent $t\in \nR_{\geq 0}$ is defined by
$$\widehat{\sI}(X,\mf{a}^t):=f_*\sO_Y(\widehat{K}_{Y/X}-J_{Y/X}-\lfloor tZ\rfloor),$$
where $\lfloor\ \ \rfloor$ means the round down of an $\nR$-divisor. Sometimes we simply write it as $\widehat{\sI}(\mf{a}^t)$.
\end{definition}

\begin{remark} (1) In \cite{Ein:MultIdeaMatherDis}, this multiplier ideal is called Mather multiplier ideal, which we have used in the first version of the paper. A similar theory also developed independently by de Fernex and Docampo in \cite{Roi:JDiscrepancy}, where they call it Jacobian multiplier ideal. Proposed by Shihoko Ishii, it is now called Mather-Jacobian multiplier ideal, or simply $MJ$-Multiplier ideal.

(2) If the variety $X$ is nonsingular, then a Mather-Jacobian multiplier ideal is simply a classical multiplier ideal. But these two notions are not necessarily equal on a $\nQ$-Gorenstein variety (cf. \cite{Ein:MultIdeaMatherDis} and \cite{Roi:JDiscrepancy}), on which both of them are defined.
\end{remark}

Now we state the asymptotic construction of Mather-Jacobian multiplier ideals. For asymptotic multiplier ideals on nonsingular varieties we refer to the book \cite{Lazarsfeld:PosAG2} or the paper \cite{Ein:UniBdSymPw}. Since the theory now has become very classical we shall be brief here.

\begin{definition} $A$ graded family of ideals $\mf{a}_\bullet=\{\mf{a}_k\}_{k\geq 1}$ on a variety $X$ is a collection of nonzero ideal sheaves $\mf{a}_k\subseteq\sO_X$ satisfying $\mf{a}_i\cdot\mf{a}_j\subseteq \mf{a}_{i+j}$ for all $i,j\geq 1$.
\end{definition}

Examples of graded family can be found in \cite[Example 1.2]{Ein:UniBdSymPw}. In particular, the collection of symbolic powers of a radical ideal is naturally a graded family of ideals.

Fix an index $l\geq 1$ in a graded family of ideals $\mf{a}_\bullet=\{\mf{a}_k\}_{k\geq 1}$ and fix a positive real number $c>0$. Consider the set of Mather-Jacobian multiplier ideals
\begin{equation}\label{eq:02}
\{\widehat{\sI}(X,\mf{a}_{pl}^{\frac{c}{p}})\}_{p\geq 1}
\end{equation}
It is standard to check (cf. \cite[Lemma 1.3.]{Ein:UniBdSymPw} or the last paragraph of \cite{Ein:MultIdeaMatherDis}) that there is a unique maximal element in the set (\ref{eq:02}), which is the so called asymptotic Mather-Jacobian multiplier ideal.

\begin{definition} Given a graded family of ideals  $\mf{a}_\bullet=\{\mf{a}_k\}_{k\geq 1}$ on a variety $X$, the {\em asymptotic Mather-Jacobian multiplier ideal} at level $l$ associated to $c>0$ is the maximal element of the set (\ref{eq:02}), and is denoted by $\widehat{\sI}(X,c\cdot||\mf{a}_{l\bullet}||)$. It is clear that for $p\gg0$, $\widehat{\sI}(X,c\cdot||\mf{a}_{l\bullet}||)=\widehat{\sI}(X,\mf{a}_{pl}^{\frac{c}{p}})$.
\end{definition}

The following proposition describes basic properties of asymptotic Mather-Jacobian multiplier ideals and extends the core technical part of \cite{Ein:UniBdSymPw}. We would not use it in the paper since we have a better Proposition \ref{p:11} for our purpose. However, we hope that it might be useful somewhere else.
\begin{proposition}\label{p:05} Let $X$ be a variety and $a_\bullet=\{\mf{a}_l\}_{l\geq 1}$ be a graded family of ideals. Let $\mf{j}_X$ be the Jacobian ideal of $X$. Then one has
\begin{itemize}
\item [(1)] for any $l\geq 1$, $\mf{j}_X\cdot\mf{a}_l\subseteq \widehat{\sI}(X,||\mf{a}_{l\bullet}||);$
\item [(2)] for any $m\geq 1$,
$$\mf{j}_X^{m-1}\cdot \widehat{\sI}(X,||\mf{a}_{ml\bullet}||)\subseteq \widehat{\sI}(X,||\mf{a}_{l\bullet}||)^m;$$
\item [(3)] suppose that $\mf{b}\subseteq\sO_X$ is an ideal such that  $\widehat{\sI}(X,||\mf{a}_{l\bullet}||)\subseteq \mf{b}$ for some fixed index $l\geq 1$, then for any $m\geq 1$,
$$\mf{j}_X^m\cdot \mf{a}_{ml}\subseteq \mf{b}^m.$$
\end{itemize}
\end{proposition}
\begin{proof} For (1), it is enough to show $\mf{j}_X\cdot\mf{a}_l\subseteq \widehat{\sI}(X,\mf{a}_l)$. Take a log resolution $f:Y\longrightarrow X$ of $\mf{j}_X\cdot \mf{a}_l$ such that $\mf{j}_X\cdot\sO_X=\sO_X({-J_{Y/X}})$ and $\mf{a}_l\cdot\sO_X=\sO_X(-Z_l)$ where $J_{Y/X}$ and $Z_l$ are effective divisors on $Y$. Since $\widehat{K}_{Y/X}$ is effective we then have
$$\mf{j}_X\cdot \mf{a}_l\subseteq \overline{\mf{j}_X\cdot \mf{a}_l}= f_*\sO_Y(-J_{Y/X}-Z_l)\subseteq f_*\sO_Y(\widehat{K}_{Y/X}-J_{Y/X}-Z_l)=\widehat{\sI}(X,\mf{a}_l),$$
where $\overline{\mf{\tau}}$ means the integral closure of an ideal $\mf{\tau}$.

For (2), let $p\gg 0$ such that
$$\widehat{\sI}(X,||\mf{a}_{l\bullet}||)=\widehat{\sI}(X,\mf{a}_{pml}^{\frac{1}{pm}}), \mbox{ and } \widehat{\sI}(X,||\mf{a}_{ml\bullet}||)=\widehat{\sI}(X,\mf{a}_{pml}^{\frac{1}{p}}).$$
We can rewrite
$$\widehat{\sI}(X,\mf{a}_{pml}^{\frac{1}{p}})=\widehat{\sI}(X,\mf{a}_{pml}^{\frac{m}{pm}})=\widehat{\sI}(X,(\mf{a}^{\frac{1}{pm}}_{pml})^m).$$
Now we have
\begin{eqnarray*}
\mf{j}_X\cdot \widehat{\sI}(X,(\mf{a}^{\frac{1}{pm}}_{pml})^m)\subseteq \widehat{\sI}(X,\mf{j}_X\cdot(\mf{a}^{\frac{1}{pm}}_{pml})^m) &=  &\widehat{\sI}(X,\mf{j}_X\cdot(\mf{a}^{\frac{1}{pm}}_{pml})^{m-1}\cdot\mf{a}^{\frac{1}{pm}}_{pml} )\\
    &\subseteq &\widehat{\sI}(X,(\mf{a}^{\frac{1}{pm}}_{pml})^{m-1})\cdot \widehat{\sI}(X,\mf{a}^{\frac{1}{pm}}_{pml})
\end{eqnarray*}
where the left-hand-side inclusion is an immediate consequence of the definition of Mather-Jacobian multiplier ideals and the right-hand-side inclusion is from the Subadditivity Theorem \cite[Theorem 3.14]{Ein:MultIdeaMatherDis}.
Finally, (2) can be proved iteratedly.

(3) is a direct consequence of (1) and (2).
\end{proof}

\begin{proposition}\label{p:11} Let $X$ be a variety and $a_\bullet=\{\mf{a}_l\}_{l\geq 1}$ be a graded family of ideals. Let $\mf{j}_X$ be the Jacobian ideal of $X$. Then for any $l\geq 1$, one has
$$\widehat{\sI}(X,\mf{j}^{m-1}_X)\cdot \mf{a}_{ml}\subseteq \widehat{\sI}(X,||\mf{a}_{l\bullet}||)^m,\quad\mbox{for }m\geq 1.$$
In particular, if $\mf{b}$ is an ideal such that $\widehat{\sI}(X,||\mf{a}_{l\bullet}||)\subseteq \mf{b}$, then for any $l\geq 1$, one has
$$\widehat{\sI}(X,\mf{j}^{m-1}_X)\cdot \mf{a}_{ml}\subseteq \mf{b}^m,\quad\mbox{for }m\geq 1.$$
\end{proposition}
\begin{proof} Take $p\gg 0$ such that
$$\widehat{\sI}(X,||\mf{a}_{l\bullet}||)=\widehat{\sI}(X,\mf{a}_{pml}^{\frac{1}{pm}}), \mbox{ and } \widehat{\sI}(X,||\mf{a}_{ml\bullet}||)=\widehat{\sI}(X,\mf{a}_{pml}^{\frac{1}{p}}).$$
Then one has
$$\widehat{\sI}(X,\mf{j}^{m-1}_X)\cdot\mf{a}_{ml}\subseteq \widehat{\sI}(X,\mf{j}^{m-1}_X\cdot \mf{a}_{ml})\subseteq \sI(X,\mf{j}^{m-1}_X\cdot \mf{a}^{\frac{1}{p}}_{pml}).$$
The first inclusion above is from the definition. The second one is because $\mf{a}^p_{ml}\subseteq \mf{a}_{pml}$ and therefore $\widehat{\sI}(X,\mf{j}^{m-1}_X\cdot \mf{a}_{ml})=\widehat{\sI}(X,\mf{j}^{m-1}_X\cdot(\mf{a}^p_{ml})^{\frac{1}{p}})\subseteq \widehat{\sI}(X,\mf{j}^{m-1}_X\cdot\mf{a}^{\frac{1}{p}}_{pml})$.

Finally, we apply Subadditivity Theorem \cite[Theorem 3.14]{Ein:MultIdeaMatherDis} $(m-1)$-times to the ideal $\widehat{\sI}(X,\mf{j}^{m-1}_X\cdot \mf{a}^{\frac{1}{p}}_{pml})=\widehat{\sI}(X,\mf{j}^{m-1}_X\cdot (\mf{a}^{\frac{1}{pm}}_{pml})^m)$ to deduce the result.
\end{proof}

\section{Geometric nullstellensatz on a projective variety}
In this section, based on Mather-Jacobian multiplier ideals we prove Theorem \ref{p:02} and give some improved results under certain singularity conditions. Our approach follows the work of \cite{Ein:GeomNull}, to which we refer the reader for further information.

We recall the distinguished subvarieties associated to an ideal sheaf. Let $X$ be a variety and $\mf{a}$ be an ideal of $\sO_X$. Let $\nu:W\longrightarrow X$ be the normalization of the blowing-up of $X$ along $\mf{a}$ so that $\mf{a}\cdot \sO_W=\sO_W(-E)$ where $E$ is an effective Cartier divisor on $W$. We can write
$$E=\sum_{i=1}^t r_iE_i$$
as a sum of distinct prime divisors $E_i$'s with some positive integer coefficients $r_i$. Write $Z_i=\nu(E_i)$ to be the image of $E_i$ on $X$ with the reduced scheme structure. Then $Z_i$'s are called the {\em distinguished subvarieties} of $\mf{a}$ with the coefficient $r_i$. The following proposition bounds the coefficients $r_i$'s of distinguished subvarieties in terms of global invariants. It was proved in the nonsingular case in \cite{Ein:GeomNull}. But it can proved for any projective variety by using the same intersection theory computation as in \cite{Ein:GeomNull}. In fact, the reader can also find a proof in \cite[Theorem 5.4.18]{Lazarsfeld:PosAG1} by setting $s=1$ in (5.20) there.

\begin{proposition}[{\cite[Proposition 3.1]{Ein:GeomNull}}]\label{p:03} Let $X$ be a projective variety with an ample line bundle $L$. Let $\mf{a}$ be an ideal sheaf of $\sO_X$ such that $\mf{a}\otimes L$ is generated by global sections. Then
$$\sum^t_{i=1}r_i\cdot\deg_LZ_i\leq \deg_LX.$$
where $Z_i$'s are distinguished subvarieties of $\mf{a}$ with coefficients $r_i$'s.
\end{proposition}

Now we prove Theorem \ref{p:02}. The basic idea is that we use a multiplier ideal as a bridge to connect an ideal with its radical.

\begin{proof}[Proof of Theorem \ref{p:02}] Let $Z_i$, $i=1,\cdots, t$, be the distinguished subvarieties of $\mf{a}$ with the coefficient $r_i$ defined by the ideal $\mf{q}_{Z_i}$.

\begin{claim}\label{cl:02} For $l\geq 0$ one has the inclusion
$$\widehat{\sI}(\sO_X)\cdot(\mf{q}^{(r_1l)}_{Z_1}\cap\cdots\cap \mf{q}^{(r_tl)}_{Z_t})\subseteq \widehat{\sI}(X,\mf{a}^l).$$
\end{claim}
\begin{proof}[Proof of claim] The inclusion is local, so we can assume that $X$ is affine. Let $\mf{j}_X$ be the Jacobian ideal of $X$ and let $f:Y\longrightarrow X$ be a log resolution of $\mf{j}_X\cdot\mf{a}$ such that $\mf{j}_X\cdot \sO_Y=\sO_Y(-J_{Y/X})$ and $\mf{a}\cdot\sO_Y=\sO_Y(-Z)$ for some effective divisors $J_{Y/X}$ and $Z$ on $Y$. Recall that by definition $\widehat{\sI}(\sO_X)=f_*\sO_Y(\widehat{K}_{Y/X}-J_{Y/X})$. For any element $g\in \widehat{\sI}(\sO_X)\cdot(\mf{q}^{(r_1l)}_{Z_1}\cap\cdots\cap \mf{q}^{(r_tl)}_{Z_t})$ we shall show that
\begin{equation}\label{eq:03}
\udiv f^*g\geq -\widehat{K}_{Y/X}+J_{Y/X}+lZ,
\end{equation}
where $\udiv f^*g$ means the effective divisor defined by $f^*g$ on $Y$. To see this let $\nu:W\longrightarrow X$ be the normalization of the blowing-up of $X$ along $\mf{a}$ so that $\mf{a}\cdot \sO_W=\sO_W(-E)$ where $E$ is an effective Cartier divisor on $W$. Write $E$ as a sum of prime divisors $E=\sum_{i=1}^t r_iE_i$ with the coefficients $r_i$'s. Note that $Z_i=\nu(E_i)$ and $f$ factors through $\nu$ via a morphism $\varphi: Y\longrightarrow W$ such that $Z=\varphi^*E$. Now we can write $g=\sum b_ic_i$ as a finite sum, where $b_i\in \widehat{\sI}(\sO_X)$ and $c_i\in \mf{q}^{(r_1l)}_{Z_1}\cap\cdots\cap \mf{q}^{(r_tl)}_{Z_t}$. It suffices to show $\udiv f^*(b_ic_i)\geq -\widehat{K}_{Y/X}+J_{Y/X}+lZ$ for each summand of $g$. Thus without loss of generality, we can assume that $g=bc$ where $b\in \widehat{\sI}(\sO_X)$ and $c\in \mf{q}^{(r_1l)}_{Z_1}\cap\cdots\cap \mf{q}^{(r_tl)}_{Z_t}$. Then it is clear that $\ord_{E_i}\nu^*c\geq r_il$ and therefore $\udiv \nu^*c\geq lE$. Thus $\udiv f^*c=\udiv \varphi^*(\nu^*c)\geq \varphi^*(lE)=lZ$. On the other hand, it is clear $\udiv f^*b\geq -\widehat{K}_{Y/X}+J_{Y/X}$. Thus we conclude that $\udiv f^*g\geq -\widehat{K}_{Y/X}+J_{Y/X}+lZ$. So the claim is proved.
\end{proof}

Now by the Skoda-type formula \cite[Theorem 3.15]{Ein:MultIdeaMatherDis}, for $l\geq n$, $\widehat{\sI}(X,\mf{a}^l)\subseteq \mf{a}$. On the other hand, if we set $r:=\max_i\{r_i\}$, then it is clear that
$$(\sqrt{\mf{a}})^{rl}\subseteq \mf{q}^{(r_1l)}_{Z_1}\cap\cdots\cap \mf{q}^{(r_tl)}_{Z_t}.$$
Thus combining the above inclusion with Claim \ref{cl:02}, for $l\geq n$, we have
$$\widehat{\sI}(\sO_X)\cdot (\sqrt{\mf{a}})^{rl}\subseteq \widehat{\sI}(\sO_X)\cdot(\mf{q}^{(r_1l)}_{Z_1}\cap\cdots\cap \mf{q}^{(r_tl)}_{Z_t})\subseteq\widehat{\sI}(X,\mf{a}^l)\subseteq \mf{a}.$$
Finally by Proposition \ref{p:03} the value  $r$ is bounded by $\deg_LX$ . The result then follows immediately if we take $l=n$.

If we further assume $X$ is normal $\nQ$-Gorenstein, then the multiplier ideal $\sI(\sO_X)$ is defined. We have the following

\begin{claim}\label{cl:02} For $l\geq 0$ one has the inclusion
$$\sI(\sO_X)\cdot(\mf{q}^{(r_1l)}_{Z_1}\cap\cdots\cap \mf{q}^{(r_tl)}_{Z_t})\subseteq \sI(X,\mf{a}^l).$$
\end{claim}
\begin{proof} Again, we assume that $X$ is affine. Take a log resolution $f:Y\longrightarrow X$ of $\mf{a}$ such that  $\mf{a}\cdot\sO_Y=\sO_Y(-Z)$ for some effective divisors $Z$ on $Y$. Recall that $\sI(\sO_X)=f_*\sO_Y(\lceil K_Y-f^*K_X\rceil)$. Take an element $g\in \sI(\sO_X)\cdot(\mf{q}^{(r_1l)}_{Z_1}\cap\cdots\cap \mf{q}^{(r_tl)}_{Z_t})$. We may assume that $g=bc$ where $b\in \sI(\sO_X)$ and $c\in \mf{q}^{(r_1l)}_{Z_1}\cap\cdots\cap \mf{q}^{(r_tl)}_{Z_t}$. As the same argument in Claim \ref{cl:02}, we can show that $\udiv f^*c\geq lZ$ and $\udiv f^*b\geq -(\lceil K_Y-f^*K_X\rceil)$. Thus we conclude that $\udiv f^*g\geq -(\lceil K_Y-f^*K_X \rceil)+lZ$. So the claim is proved.
\end{proof}
Now using Skoda-type formula and the above claim, for $l\geq n$, we have
$$\sI(\sO_X)\cdot (\sqrt{\mf{a}})^{rl}\subseteq \sI(\sO_X)\cdot(\mf{q}^{(r_1l)}_{Z_1}\cap\cdots\cap \mf{q}^{(r_tl)}_{Z_t})\subseteq\sI(X,\mf{a}^l)\subseteq \mf{a}.$$
The rest is then clear.

\end{proof}

\begin{corollary}\label{p:07} Let $X$ be a projective variety of dimension $n$ with an ample line bundle $L$. Let $\mf{j}_X$ be the Jacobian ideal of $X$ and assume that $\mf{j}_X\otimes L$ is globally generated. Then one has
$$(\sqrt{\mf{j}_X})^{(n+1)\cdot \deg_LX}\subseteq \mf{j}_X.$$
\end{corollary}
\begin{proof} We follow the same notation in the proof of Theorem \ref{p:02}. But this time let $Z_i$, $i=1,\cdots, t$, be the distinguished subvarieties of $\mf{j}_X$ with the coefficient $r_i$ defined by the ideal $\mf{q}_{Z_i}$. We claim first that for $l\geq 1$,
$$\mf{q}^{(r_1l)}_{Z_1}\cap\cdots\cap \mf{q}^{(r_tl)}_{Z_t}\subseteq \widehat{\sI}(X,\mf{j}_X^{l-1}).$$
To see this, we can assume $X$ is affine. Then for any $g\in\mf{q}^{(r_1l)}_{Z_1}\cap\cdots\cap \mf{q}^{(r_tl)}_{Z_t}$, we have $\udiv f^*g\geq lJ_{Y/X}$, which means that $\udiv f^*g\geq -\widehat{K}_{Y/X}+J_{Y/X}+(l-1)J_{Y/X}$. Thus the claim is true.

Set $r:=\max_i\{r_i\}$ and use the Skoda-type formula \cite[Theorem 3.15]{Ein:MultIdeaMatherDis}, for $l\geq n+1$, we have inclusions
$$(\sqrt{\mf{j}_X})^{rl}\subseteq \mf{q}^{(r_1l)}_{Z_1}\cap\cdots\cap \mf{q}^{(r_tl)}_{Z_t}\subseteq \widehat{\sI}(X,\mf{j}_X^{l-1})\subseteq \mf{j}_X.$$
Take $l=n+1$  and use Proposition \ref{p:03} to bound $r$ by $\deg_LX$. The result then follows.
\end{proof}

Once we put strong singularity conditions on the variety $X$, we actually can remove $\widehat{\sI}(\sO_X)$ or $\sI(\sO_X)$ to get a simple formula.
Recall that a variety $X$ is called {\em $MJ$-canonical} (Mather-Jacobian canonical, cf. \cite{Ein:MultIdeaMatherDis} or \cite{Roi:JDiscrepancy}) if $\widehat{K}_{Y/X}-J_{Y/X}$ is effective in a log resolution  $f:Y\longrightarrow X$  of $\mf{j}_X$. It has been showed in \cite{Ein:MultIdeaMatherDis} and \cite{Roi:JDiscrepancy} that a $MJ$-canonical variety has rational singularities. In particular, if $X$ has $MJ$-canonical, then $\widehat{\sI}(\sO_X)=\sO_X$. Thus we have the following corollary.

\begin{corollary}\label{p:08} Let $X$ be a projective variety of dimension $n$ with an ample line bundle $L$. Let $\mf{a}$ an ideal sheaf of $\sO_X$. Assume that $X$ is $MJ$-canonical and $\mf{a}\otimes L$ is globally generated. Then one has
$$(\sqrt{\mf{a}})^{n\cdot\deg_LX}\subseteq \mf{a}.$$
\end{corollary}

A normal $\nQ$-Gorenstein variety $X$ is called {\em log terminal} if in a log resolution $f:Y\longrightarrow X$, the coefficients in the relative canonical divisor $K_{Y/X}:=K_Y-f^*K_X$ are all $>-1$. The multiplier ideal $\sI(\sO_X)$ is then trivial, i.e., $\sI(\sO_X)=\sO_X$. So we deduce the following corollary immediately.

\begin{corollary}\label{p:09} Let $X$ be a projective log terminal variety of dimension $n$ with an ample line bundle $L$. Let $\mf{a}$ be an ideal sheaf of $\sO_X$. Assume that $\mf{a}\otimes L$ is globally generated. Then one has
$$(\sqrt{\mf{a}})^{n\cdot\deg_LX}\subseteq \mf{a}.$$
\end{corollary}

\begin{remark}\label{rmk:01} (1) If $X$ is nonsingular then the formula in Theorem \ref{p:02} is the one proved in \cite{Ein:GeomNull}.

(2) Theorem \ref{p:02} works for arbitrary projective varieties but at the price that we have to insert $\widehat{\sI}(\sO_X)$ into the formula.  It would be interesting to know if one can really remove $\widehat{\sI}(\sO_X)$.

(3) In the first version of the paper, we showed the formula
$$\mf{j}_X\cdot(\sqrt{\mf{a}})^{n\cdot\deg_LX}\subseteq \mf{a}.$$
Now it is a direct consequence of Theorem \ref{p:02} since we have inclusion $\mf{j}_X\subseteq\overline{\mf{j}}_X\subseteq \widehat{\sI}(\sO_X)$.
\end{remark}

\section{Comparison of symbolic and ordinary
powers of ideals}

In this section, we prove the Theorem \ref{p:04} on comparison of symbolic powers and ordinary powers of a radical ideal. As stated in Introduction, we basically follow the idea of \cite{Ein:UniBdSymPw} but using asymptotic Mather-Jacobian multiplier ideals. The crucial point is to relate the asymptotic Mather-Jacobian multiplier ideals of symbolic powers to ordinary powers. In order to make the proof clean and transparent, we only deal with the basic case and leave the general case in a remark.

Let $X$ be a variety and $Z$ be a reduced subscheme defined by an ideal $\mf{q}$. The collection $\{\mf{q}^{(k)}\}_{k\geq 1}$ of symbolic powers of $\mf{q}$ is a graded family of ideals. So for a fixed $l\geq 1$, we obtain the asymptotic Mather-Jacobian multiplier ideal $\widehat{\sI}(X,||\mf{q}^{(l\bullet)}||)$ at level $l$. The following lemma plays the central role in the proof of Theorem \ref{p:04}. We first give an elementary proof which can be easily used in the general case. Then we give an alternative simple proof suggested by the referee, which involves deep results in the Mather-Jacobian singularity theory.

\begin{lemma}\label{p:01} Let $X$ be a variety and $\mf{q}$ be an ideal sheaf defining a proper reduced subscheme $Z$ of $X$ of codimension $\leq e$. Then
$$\widehat{\sI}(X, ||\mf{q}^{(e\bullet)}||)\subseteq \mf{q}.$$
\end{lemma}
\begin{proof} Let $\eta$ be one generic point of $Z$. It suffices to show that $\widehat{\sI}(X, ||\mf{q}^{(e\bullet)}||)_{\eta}\subseteq \mf{q}_\eta$ since $\mf{q}$ is radical. Thus by working locally we can assume that $Z$ is integral with only one generic point $\eta$. And we shall show that there is an open neighborhood of $\eta$ on which $\widehat{\sI}(X, ||\mf{q}^{(e\bullet)}||)\subseteq \mf{q}$.

To this end, first of all, by taking an open neighborhood of $\eta$, we can assume that $X$ is a subvariety of a nonsingular variety $A$ of codimension $c$. Second, we can take an open set $U$ in $A$ containing $\eta$ such that
\begin{itemize}
\item [(1)] $Z|_U$ is nonsingular;
\item [(2)] If we write $U_X:=X\cap U$ to be an open set of X, then on $U_X$ the symbolic power $\mf{q}^{(t)}$ is the same as the ordinary power $\mf{q}^t$, i.e., for all $t\geq 1$, $\mf{q}^{(t)}|_{U_X}=\mf{q}^t|_{U_X}$.
\end{itemize}
(1) is clear because $Z$ is generically nonsingular. (2) is because there are only finitely many embedded associated primes of $\mf{q}^t$ for all $t\geq 1$ so we can choose the open set $U$ to avoid those primes which are certainly not $\mf{q}$.

Finally, we replace $A$ by this open set $U$. We denote by $\widetilde{\mf{q}}$ as the defining ideal of $Z$ in $A$ and by $I_X$ as the defining ideal of $X$ in $A$. Clearly, $I_X\subseteq \widetilde{\mf{q}}$ and $\widetilde{\mf{q}}\cdot \sO_X=\mf{q}$. By the construction above we see that $\widehat{\sI}(X, ||\mf{q}^{(e\bullet)}||)=\widehat{\sI}(X, \mf{q}^e)$ and thus in the following we shall show $\widehat{\sI}(X,\mf{q}^e)\subseteq \mf{q}$.

Now let $\mu: A'\longrightarrow A$ be a log resolution of $\widetilde{\mf{j}}_X\cdot \widetilde{\mf{q}}$ where $\widetilde{\mf{j}}_X$ is the pull back of $\mf{j}_X$ in $A$, so that we have $\widetilde{\mf{q}}\cdot \sO_{A'}=\sO_{A'}(-E)$ and $\widetilde{\mf{j}}_X\cdot\sO_{A'}=\sO_{A'}(-J)$. Let $X'$ be the strict transform of $X$ in $A'$. Then by \cite[Lemma 2.11]{Ein:MultIdeaMatherDis}, we obtain a birational morphism $\varphi':\overline{A}\longrightarrow A'$ such that the composition morhpism $\varphi:=\mu\circ \varphi'$ is a factorizing resolution of $X$ inside $A$ (for definition of factorizing resolution see \cite[Definition 2.10]{Ein:MultIdeaMatherDis}) and $\overline{X}\cup \Exc(\varphi)\cup \Supp \varphi'^*E\cup\Supp \varphi'^*J$ has simple normal crossings, where $\overline{X}$ is the strict transform of $X$ in $\overline{A}$. The restriction morphism $\varphi|_{\overline{X}}:\overline{X}\longrightarrow X$ then factors through the Nash blow-up of $X$. Furthermore, by the definition of factorizing resolution, we can write $I_X\cdot\sO_{\overline{A}}=I_{\overline{X}}\cdot \sO_{\overline{A}}(-R_{X/A})$ where $R_{X/A}$ is an effective divisor on $\overline{A}$ supported on the exceptional locus of $\varphi$.

\begin{claim}\label{cl:01} One has $\sO_{\overline{A}}(-R_{X/A})\subseteq \sO_{\overline{A}}(-\varphi'^*E)$.
\end{claim}
\begin{proof}[Proof of Claim] We write $\varphi'^*E=\sum r_iE_i$ as a sum of prime divisors of $\overline{A}$ with certain coefficients. Since $I_X\subseteq \widetilde{\mf{q}}$ we have by our construction
\begin{equation}\label{eq:01}
I_{\overline{X}}\cdot\sO_{\overline{A}}(-R_{X/A})\subset \sO_{\overline{A}}(-\sum r_iE_i).
\end{equation}
Since $E_i$'s are simple normal crossings, we have $\sO_{\overline{A}}(-\sum r_iE_i)=\cap\sO_{\overline{A}}(-r_iE_i)$. Thus it is enough to show $\sO_{\overline{A}}(-R_{X/A})\subseteq \sO_{\overline{A}}(-r_iE_i)$ for each $E_i$.
We prove this at each closed point of $\overline{A}$. For this, let $y\in \overline{A}$ be a closed point and assume that $(\sO_{\overline{A},y},\mf{m}_y)=(B,\mf{m})$ is the local ring at $y$. Then at the point $y$, we can assume that $\mf{m}=(x_1,\cdots, x_n)$ is generated by regular sequence where $n=\dim \overline{A}$, and since $\overline{X}$ meets $E_i$ as simple normal crossings, we can assume $I_{\overline{X},y}=(x_1,\cdots, x_c)$ and $E_i$ is generated by $x_{c+1}$. Then the ideal $J:=\sO_{\overline{A}}(-r_iE_i)_y=(x^{r_i}_{c+1})$ is a $(x_{c+1})$-primary ideal. Assume that $\sO_{\overline{A}}(-R_{X/A})_y=(f)$ for some element $f\in B$. Then by \ref{eq:01}, we see $x_1\cdot f\in J$. But in any case no powers of $x_1$ can be contained in $J$ thus $f$ must be in $J$ since $J$ is a primary ideal. This shows that $\sO_{\overline{A}}(-R_{X/A})_y\subseteq \sO_{\overline{A}}(-r_iE_i)_y$. Thus we conclude that $\sO_{\overline{A}}(-R_{X/A})\subseteq \sO_{\overline{A}}(-r_iE_i)$ and therefore the claim follows.
\end{proof}

Now by \cite[Lemma 2.12]{Ein:MultIdeaMatherDis} we have $K_{\overline{A}/A}-cR_{X/A}|_{\overline{X}}=\widehat{K}_{\overline{X}/X}-J_{\overline{X}/X}$. Write $D=K_{\overline{A}/A}-cR_{X/A}-e\varphi'^*E$ and consider the exact sequence
$$0\longrightarrow I_{\overline{X}}\cdot \sO_{\overline{A}}(D)\longrightarrow \sO_{\overline{A}}(D)\longrightarrow \sO_{\overline{X}}(D|_{\overline{X}})\longrightarrow 0.$$
As showed in \cite[Proof of Proposition 2.13]{Ein:MultIdeaMatherDis} by using a spectral sequence argument we deduce that $R^1\varphi_*I_{\overline{X}}\cdot \sO_{\overline{A}}(D)=0$. Also notice that $\varphi_*\sO_{\overline{X}}(D|_{\overline{X}})=\widehat{\sI}(X,\mf{q}^e)$. Thus pushing down this sequence via $\varphi$, we have a surjective morphism
$$\varphi_*\sO_{\overline{A}}(D)\longrightarrow \widehat{\sI}(X,\mf{q}^e)\longrightarrow 0.$$
This means that $\varphi_*\sO_{\overline{A}}(D)\cdot\sO_X=\widehat{\sI}(X,\mf{q}^e)$.
On the other hand, by Claim (\ref{cl:01}), we have
$$\sO_{\overline{A}}(D)\subseteq \sO_{\overline{A}}(K_{\overline{A}/A}-(c+e)\varphi'^*E).$$
Note that $\varphi_*\sO_{\overline{A}}(K_{\overline{A}/A}-(c+e)\varphi'^*E)=\sI(A,(c+e)Z)$, the usual multiplier ideal of $Z$. Since $Z$ has codimension $\leq (c+e)$ in $A$, we then see that $\sI(A,(c+e)Z)\subseteq \widetilde{\mf{q}}$. Thus $\varphi_*\sO_{\overline{A}}(D)\subseteq \widetilde{\mf{q}}$ and therefore we conclude that $\widehat{\sI}(X,\mf{q}^e)\subseteq \widetilde{\mf{q}}\cdot \sO_X=\mf{q}$.
\end{proof}

Now we are ready to prove Theorem \ref{p:04}.

\begin{proof}[Proof of Theorem \ref{p:04}] By Lemma \ref{p:01}, we have $\widehat{\sI}(X, ||\mf{q}^{(e\bullet)}||)\subseteq \mf{q}$. Then apply Proposition \ref{p:05} we immediately get the desired result.
\end{proof}

\begin{remark}\label{rmk:02} We can generalize Theorem \ref{p:04} to higher symbolic powers easily. We state the result here and leave the details to the reader.

(1) Lemma \ref{p:01} can be generalized as follows: Let $X$ be a variety and $\mf{q}$ be an ideal sheaf defining a proper reduced subscheme $Z$ of $X$ of codimension $\leq e$. Then for an integer $l\geq e$, $\widehat{\sI}(X, ||\mf{q}^{(l\bullet)}||)\subseteq \mf{q}^{(l+1-e)}$. Consequently, we can generalize Theorem \ref{p:04} as for an integer $l\geq e$, one has $$\widehat{\sI}(\mf{j}^{m-1}_X)\cdot \mf{q}^{(ml)}\subseteq (\mf{q}^{(l+1-e)})^m,$$
for all $m\geq 1$.

(2) We also can consider unmixed ideals (cf. \cite[3. Generalization]{Ein:UniBdSymPw}) to get the following result analogue to \cite[Variant]{Ein:UniBdSymPw}. Let $\mf{q}$ be an unmixed ideal on a variety $X$, and assume that every associated subvariety of $\mf{q}$ has codimension $\leq e$. Then for an integer $l\geq e$, one has
$$\widehat{\sI}(\mf{j}^{m-1}_X)\cdot \mf{q}^{(me)}\subseteq \mf{q}^m,$$
for all $m\geq 1$.
\end{remark}

The Lemma \ref{p:01} has a simple but very interesting proof suggested by the referee. It requires knowledge about Mather-Jacobian minimal log discrepancy and its inversion of adjunction formula. The reader may find this topic in \cite{Roi:JDiscrepancy} and \cite{Ishii:MatherDis}. Note that here we use notation $\widehat{\mld}$ means the  Mather-Jacobian minimal log discrepancy, that is we need to consider ``Mather discrepancy divisor - Jacobian divisor" when we compute the discrepancy. It is also called Jacobian minimal log discrepancy in \cite{Roi:JDiscrepancy}.

\begin{proposition}\label{p:10} Let $X$ be a variety and $\mf{a}\subseteq \sO_X$ be an ideal. Let $Z\subset X$ be a proper subvariety defined by the ideal $\mf{q}_Z$ and let $\eta_Z$ be the generic point of $Z$ in $X$.
\begin{itemize}
\item [(1)] For any $t\in \nR_{\geq 0}$, if $\widehat{\mld}_{\eta_Z}(X,\mf{a}^t)\leq 0$, then one has $\widehat{\sI}(X,\mf{a}^t)\subseteq \mf{q}_Z$.
\item [(2)] Assume that $\mf{q}_Z\subseteq \widehat{\sI}(\sO_X)$ (e.g. $X$ is $MJ$-canonical), then for $t\in \nR_{\geq 0}$ one has
    $$\widehat{\mld}_{\eta_Z}(X,\mf{q}_Z^t)\leq 0 \mbox{ if and only if }\widehat{\sI}(X,\mf{q}_Z^t)\subseteq \mf{q}_Z$$
\end{itemize}

\end{proposition}
\begin{proof} (1) Since $\widehat{\mld}_{\eta_Z}(X,\mf{a}^t)\leq 0$, we have a log resolution $f:Y\longrightarrow X$ of $\mf{j}_X\cdot \mf{a}\cdot\mf{q}_Z$ such that $\mf{j}_X\cdot \sO_Y=\sO_Y(-J_{Y/X})$, $\mf{a}\cdot\sO_Y=\sO_Y(-E)$ and there is a prime divisor $F$ on $Y$ with $f(F)=Z$ and its Mather-Jacobian log discrepancy
$$\widehat{a}_F:=\ord_F(\widehat{K}_{Y/X})-\ord_FJ_{Y/X}-\ord_FE+1\leq 0.$$
Thus if we write
$$\widehat{K}_{Y/X}-J_{Y/X}-\lfloor tE\rfloor=P-N$$
where $P$ and $N$ are effective divisors without common components, then the divisor $F$ should be in the support of $N$. Hence we have
$$\widehat{\sI}(X,\mf{a}^t)=f_*\sO_Y(-N)\subseteq f_*\sO_Y(-E)\cap \sO_X.$$
But since $E$ dominates $Z$, we can check that $f_*\sO_Y(-E)\cap \sO_X=\mf{q}_Z$.

(2) The ``only if" part is from (1). So we prove ``if" part by assuming $\widehat{\sI}(X,\mf{q}_Z^t)\subseteq \mf{q}_Z$. Take a log resolution of $f:Y\longrightarrow X$ of $\mf{j}_X\cdot\mf{q}_Z$ such that $\mf{j}_X\cdot \sO_Y=\sO_Y(-J_{Y/X})$, $\mf{q}_Z\cdot\sO_Y=\sO_Y(-E)$. We write
$\widehat{K}_{Y/X}-J_{Y/X}-\lfloor tE\rfloor=P-N$
where $P$ and $N$ are effective divisors without common components. Write $N=\sum_{i=1}^s r_iE_i$ as a sum of distinct prime divisors with positive coefficients. Then notice that
\begin{equation}\label{eq:04}
\widehat{\sI}(X,\mf{q}_Z^t)=f_*\sO_Y(-N)=f_*\sO_Y(-r_1E_1)\cap\cdots\cap f_*\sO_Y(-r_sE_s).
\end{equation}
Each $f_*\sO_Y(-r_iE_i)\cap \sO_X$ is a primary ideal with a radical $f_*\sO_Y(-E_i)\cap \sO_X$. (Since $\widehat{\sI}(X,\mf{q}_Z^t)$ is already an ideal so we do not need to intersect with $\sO_X$ in the equation (\ref{eq:04}).) Hence the decomposition of (\ref{eq:04}) in fact induces a primary decomposition of $\widehat{\sI}(X,\mf{q}_Z^t)$. Now since we assume $\widehat{\sI}(X,\mf{q}_Z^t)\subseteq \mf{q}_Z$, there must be one $E_i$, say $E_1$, such that $f_*\sO_Y(-E_1)\cap \sO_X\subseteq \mf{q}_Z$. There are two possibilities. The first case is that  $E_1$ is in the support of $\widehat{K}_{Y/X}-J_{Y/X}$ with a negative coefficient. Then $f_*\sO_Y(-E_1)\cap \sO_Y$ must be an associated prime of the ideal $\widehat{\sI}(\sO_X)$. But we assume $\mf{q}_Z\subseteq \widehat{\sI}(\sO_X)$. Hence $\mf{q}_Z=f_*\sO_Y(-E_1)\cap \sO_X$. The second case is that $E_1$ is in the support of $E$. But $E$ is the support of the preimage of $Z$. Hence we have $\mf{q}_Z=f_*\sO_Y(-E_1)\cap \sO_X$. So in any case, we find that $E_1$ dominates $Z$ with negative coefficient in $\widehat{K}_{Y/X}-J_{Y/X}-\lfloor tE\rfloor$. We conclude that $\widehat{\mld}_{\eta}(X,\mf{q}_Z^t)\leq 0$.
\end{proof}

We shall use the following version of the inversion of adjunction theorem proved in \cite[Theorem 4.10]{Roi:JDiscrepancy} and \cite[Theorem 3.10]{Ishii:MatherDis}).

\begin{proposition}\label{p:12}Let $X$ be a subvariety of a nonsingular variety $A$ of codimension $c$. Let $\widetilde{\mf{a}}\subset \sO_A$ be an ideal such that $\mf{a}:=\widetilde{\mf{a}}\cdot \sO_X$ is a nonzero ideal of $\sO_X$. Let $\eta\in X$ be a point such that $W:=\overline{\{\eta\}}$ is a proper closed subset of $X$. Then for any $t\in \nR_{\geq 0}$ one has
$$\widehat{\mld}_{\eta}(X,\mf{a}^t)=\mld_{\eta}(A,\widetilde{\mf{a}}^t\cdot I^c_X)$$
\end{proposition}
\begin{proof} We can find a small open set $U$ of $\eta$ in $A$ such that $\mld_{\eta}(A,\widetilde{\mf{a}}^t\cdot I^c_X)=\mld_{W\cap U}(U,(\widetilde{\mf{a}}^t\cdot I^c_X)|_U)$ and $\widehat{\mld}_{\eta}(X,\mf{a}^t)=\widehat{\mld}_{W\cap U}(X\cap U,\mf{a}^t|_{X\cap U})$. For example, we can first fix a log resolution of $\widetilde{\mf{a}}^t\cdot I^c_X$ and then remove the centers of prime divisors appeared in the resolution, which are properly contained in $W$. Then we apply the inversion of adjunction \cite[Theorem 4.10]{Roi:JDiscrepancy} or \cite[Theorem 3.10]{Ishii:MatherDis}) on this open set $U$ to deduce the result.
\end{proof}

\begin{proof}[Alternative proof of Lemma \ref{p:01}] Using the same setting and reduction in the proof of Lemma \ref{p:01}, we just need to show $\widehat{\sI}(X,\mf{q}^e)\subseteq \mf{q}$.

Let $\eta$ be the generic point of $Z$. By using Proposition \ref{p:12}, we have $\widehat{\mld}_{\eta}(X,\mf{q}^e)=\mld_{\eta}(A,I^c_X\cdot\widetilde{\mf{q}}^e)$. Now take a log resolution of $I_X\cdot\widetilde{\mf{q}}$ as $f:Y\longrightarrow X$ and let $F$ be the prime divisor coming from the blowing-up of $A$ along $\widetilde{\mf{q}}$ (note that by reduction we assume $Z$ is nonsingular in $A$). Now it is an easy computation that the log discrepancy of $F$ is $\leq 0$. Hence we have $\widehat{\mld}_{\eta}(X,\mf{q}^e)\leq 0$ and then Proposition \ref{p:10} can be applied.
\end{proof}

\bibliographystyle{alpha}

\end{document}